\documentclass[a4paper,12pt]{amsart}

\usepackage[utf8]{inputenc}
\usepackage{amsmath, amsthm, amsfonts, amssymb}
\usepackage[foot]{amsaddr}
\usepackage{stmaryrd}	
\usepackage{graphicx}

\usepackage{hyperref} 

\renewcommand{\(}{\left (}
\renewcommand{\)}{\right )}
\newcommand{\norm}[1]{\left\Vert#1\right\Vert}

\renewcommand{\phi}{\varphi}

\newcommand{\N}{\mathbb{N}}

\newcommand{\R}{\mathbb{R}}
\newcommand{\C}{\mathbb{C}}
\newcommand{\II}{\mathbb{I}}

\newcommand{\scalar}[2]{\left\langle#1,#2\right\rangle}
\newcommand{\co}{\colon}

\newcommand{\E}{\mathbb{E}}
\renewcommand{\P}{\mathbb{P}}

\newcommand{\mc}{\mathcal}
\newcommand{\be}{\begin{equation}}
\newcommand{\ee}{\end{equation}}

\newtheorem{theorem}{Theorem}
\newtheorem{lemma}[theorem]{Lemma}
\newtheorem{corollary}[theorem]{Corollary}
\newtheorem{prop}[theorem]{Proposition}
\theoremstyle{definition}
\newtheorem{remark}[theorem]{Remark}
\newtheorem{example}[theorem]{Example}

\DeclareMathOperator{\Id}{Id}
\DeclareMathOperator{\diag}{diag}
\DeclareMathOperator{\Span}{span}
\DeclareMathOperator{\Tr}{Tr}

\title[A sharp upper bound for sampling numbers]{
A sharp upper bound for \\sampling numbers in $\mathbf L_{\mathbf 2}$}

\author{
Matthieu Dolbeault$^1$
\and
David Krieg$^{2,3}$
\and
Mario Ullrich$^{2}$
}

\address{$^1$Sorbonne Universit\'e, Paris, France}
\address{$^2$Institut f\"ur Analysis, 
Johannes Kepler Universit\"at Linz, Austria}
\address{$^3$RICAM, 
Austrian Academy of Sciences, Linz, Austria.}
\email{
\quad\parbox{0.7\linewidth}{
matthieu.dolbeault@sorbonne-universite.fr,
david.krieg@jku.at, 
mario.ullrich@jku.at
}
}
\thanks{DK is supported by the Austrian Science Fund 
Project F5506, 
	which is part of the Special Research Program 
	\emph{Quasi-Monte Carlo Methods: Theory and Applications}.}


\begin{document}

\begin{abstract}
For a class $F$ of complex-valued functions on a set~$D$, 
we denote by $g_n(F)$ its sampling numbers, i.e.,
the minimal worst-case error on $F$, 
measured in $L_2$, that can be achieved 
with a recovery algorithm based on $n$ function evaluations. 
We prove that there is a universal constant $c\in\N$ such that,
if $F$ is the unit ball of a separable reproducing kernel Hilbert space,
then
\[
g_{cn}(F)^2 \,\le\, \frac{1}{n}\sum_{k\geq n} d_k(F)^2,
\]
where $d_k(F)$ are the Kolmogorov widths (or approximation numbers) 
of $F$ in $L_2$.
We also obtain similar upper bounds 
for more general classes $F$, including all compact subsets of 
the space of continuous functions on a bounded domain $D\subset \R^d$, 
and 
show that these bounds are sharp by providing examples 
where the converse inequality holds
up to a constant.
The results rely on the solution to the Kadison-Singer problem, 
which we extend to the subsampling of a sum of infinite rank-one matrices.
\end{abstract}

\maketitle

\section{Introduction and main results} 

The general question of how well point-wise 
evaluations perform for approximating a function, 
which is often called \emph{sampling recovery} or 
approximation using \emph{standard information}, is 
a classical question in theoretical and applied mathematics.
A historical treatment and various basics may be found in
the monographs \cite{CS,DL, DTU, Tem18, Wen} for general approximation theory and in \cite{NW1,NW2,NW3} for information-based complexity.
It is of particular interest to compare the \emph{power of function evaluations}
with the power of optimal linear measurements (which could be Fourier coefficients or derivatives), 
since the latter are well understood in many cases 
and easier to handle from a theoretical point of view, 
while the first are of larger practical relevance.
The quest for a systematic comparison 
has attracted much attention recently. 
We will describe the history and related results below 
after presenting the setting and the main results, 
see also Section~\ref{biblio}.

The \emph{power} of a given class of measurements 
is often expressed in terms of the minimal error achievable 
with a given amount of such information. 
Here, we consider $L_2$-approximation in a worst-case setting, 
so that these minimal errors correspond to sampling numbers 
and Kolmogorov (or approximation) numbers, as we summarize below.

Let $(D,\mc A, \mu)$ be a measure space and
$L_2:=L_2(D,\mc A, \mu)$ be the space of square-integrable 
complex-valued functions on $D$. 
Let $F$ be a set of functions contained in~$L_2$.
The \emph{Kolmogorov widths} of $F$ in $L_2$ are defined by
\[
d_k(F)\,:=\,\inf_{\substack{\ell_1,\dots,\ell_k\colon F \to \C \\ \varphi_1,\dots,\varphi_k\in L_2}}\, 
\sup_{f\in F}\, 
\Big\|f - \sum_{i=1}^k \ell_i(f)\, \varphi_i\Big\|_{L_2}.
\]
This is the worst-case error of an optimal approximation 
within a linear space of dimension $k$. 
It coincides with the $k$th approximation number (or linear width) of $F$,
which is the worst-case error of an optimal linear algorithm that uses 
at most $k$ linear functionals as information,
see Remark~\ref{rem:Kolmogorov}.
On the other hand, the \emph{sampling numbers} are given by
\[
g_n(F) \,:=\,
\inf_{\substack{x_1,\dots,x_n\in D\\ \varphi_1,\dots,\varphi_n\in L_2}}\, 
\sup_{f\in F}\, 
\Big\|f - \sum_{i=1}^n f(x_i)\, \varphi_i\Big\|_{L_2},
\]
i.e., $g_n(F)$ is
the minimal worst-case error of linear algorithms based on $n$ function evaluations.
Therefore, the task is to compare the numbers $d_k(F)$ and $g_n(F)$.

It is clear that we have $g_n(F)\ge d_n(F)$. 
Here, we aim 
for an upper bound of $g_n(F)$
in terms of the numbers $d_k(F)$.
We first describe the situation
where $F$ is the unit ball of a separable reproducing kernel Hilbert space (RKHS).
A priori, it is not clear whether such a bound is even possible.
And indeed, there can be no such bound in the case that 
$(d_k(F))\notin\ell_2$.
More precisely, it is shown in \cite{HNV} that for any non-negative 
and non-increasing sequence 
$(\sigma_k)\notin\ell_2$
and any sequence $(\tau_n)$ tending to infinity, e.g.\ $\tau_n=\log \log n$, 
there exists a RKHS
with unit ball $F$ such that $d_k(F)=\sigma_k$ for all $k$ but 
$\limsup_{n\to\infty} \tau_n\cdot g_n(F)>0$. 

The situation is completely different when $(d_k(F))\in\ell_2$,
which is equivalent 
to assuming that the kernel $K$ of the Hilbert space has finite trace
\begin{equation}\label{eq:trace}
\int_DK(x,x)\,d\mu(x)<\infty,
\end{equation}
see, e.g.,~\cite{MU}.
Under this assumption, first upper bounds on $g_n(F)$ in terms of 
the numbers $d_k(F)$ were obtained more than 20 years ago in~\cite{WW01}.
These upper bounds were later improved in \cite{KU1,KWW,NSU}. 
On the other hand, a lower bound from \cite[Theorem~2]{HNKV}
tells us how far these improvements might go:
for every non-negative and non-increasing $(\sigma_k)\in\ell_2$, 
there exists a separable RKHS
with unit ball~$F$ such that $d_k(F)=\sigma_k$ for all $k\in\N$ and
\begin{equation}\label{eq:lower}
g_{\lfloor m/8\rfloor}(F) \,\geq\, \sqrt{\frac{1}{m}\sum_{k\geq m}d_k(F)^2}
\end{equation}
for infinitely many values of $m\in\N$.  
And indeed, it turns out that this is already the worst possible scenario.
The main result of this paper is
an upper bound, which matches the above lower bound \eqref{eq:lower} 
up to a universal constant,
and which is true for any separable reproducing kernel Hilbert space.

\begin{theorem}
\label{main_theorem}
There is a universal constant $c\in\N$ such that the following holds.
Let $\mu$ be a measure on a set $D$ 
and let $F\subset L_2(\mu)$ be the unit ball of a separable RKHS on $D$   
such that the finite trace assumption~\eqref{eq:trace} holds.
Then, for all $m\in\N$, we have
\begin{equation*}
g_{cm}(F) \,\le\, \sqrt{\frac{1}{m}\sum_{k\geq m}d_k(F)^2}.
\end{equation*}
\end{theorem}

This settles the question on the power of standard information 
compared to general linear information
for the problem of $L_2$-approximation on Hilbert spaces,
and solves the open problems from \cite{HNKV, KU1}, 
Open Problem 140 in \cite{NW3},
as well as Outstanding Open Problem~1.4 in~\cite{DTU} for $L_2$-approximation.
The latter is discussed in Example~\ref{ex:tp}, where 
we consider tensor product spaces.
We note that the case of $L_p$-approxi\-ma\-tion ($p\neq2$) 
is widely open. A slightly stronger version of Theorem~\ref{main_theorem}
and explicit constants are given in Theorem~\ref{main_theorem_local}.

Let us add that, in principle, 
Theorem~\ref{main_theorem} does
only imply the \emph{existence} of (linear) sampling algorithms 
achieving the error bound.
However, 
all upper bounds on $g_n(F)$ will be obtained by a suitable 
(unregularized)
\emph{least squares method}, see Remark~\ref{rem:alg} and Section~\ref{sec:proof}.

Theorem~\ref{main_theorem} is a direct continuation of the series of 
works initiated in \cite{KU1}, 
in which the sampling numbers were bounded by 
\[
g_{\lfloor c\,m \log m\rfloor}(F) \,\le\, \sqrt{\frac{1}{m}\sum_{k\geq m}d_k(F)^2 },
\]
see also \cite{KUV,U2020}, 
and an improvement from \cite{NSU}, where the 
logarithmic oversampling was removed
in exchange for an additional factor $\sqrt{\log m}$ on the right hand side.

The ingredients for the proof are still the existence of
good point sets with 
$\mathcal O(m\log m)$ points from~\cite{KU1}, 
and a subsampling of $\mathcal O(m)$ points based on
the solution to the Kadison-Singer problem \cite{MSS}.
The Kadison-Singer subsampling 
has already been applied for the related problem
of sampling discretization in \cite{LT}
(see \cite{KKLT} for a survey) 
and was subsequently introduced to the study of sampling numbers 
in \cite{NSU,T20}.
In these papers,
the subsampling was, roughly speaking, 
only performed for a finite-dimensional sub-problem 
which resulted in the excessive factor $\sqrt{\log m}$ in \cite{NSU}.
The new ingredient here is an infinite-dimensional version 
of the subsampling theorem that might be of independent interest, 
see Proposition~\ref{prop:subsampling}.

If we apply Theorem~\ref{main_theorem}
and the lower bound from \cite{HNKV} to 
sequences with polynomial decay, we obtain the following 
characterization.

\begin{corollary}\label{cor:Hilbert}
Let $F$ be the unit ball of a 
separable RKHS with 
\begin{equation}\label{eq:UBd}
d_n(F) \,\lesssim\, n^{-\alpha} \log^{-\beta} n
\end{equation}
for some $\alpha\ge 1/2$, $\beta\in\R$ and $c>0$. Then
\begin{equation}\label{eq:UBg}
g_n(F) \,\lesssim\, \begin{cases} n^{-\alpha} \log^{-\beta} n
\quad &\text{if}\ \alpha>1/2,\\[2pt]
n^{-\alpha} \log^{-\beta+1/2} n
\quad &\text{if}\ \alpha=1/2 \text{ and } \beta>1/2.
\end{cases}
\end{equation}
Moreover, there exist classes $F$ such that these bounds are sharp. 
\end{corollary}

Here, $a_n \lesssim b_n$ means that there is a constant $c>0$
such that $a_n \le c b_n$ for all but finitely many 
$n\in \N$; later we will also use 
the symbols $\gtrsim$ and $\asymp$ which are defined accordingly.
It is clear from Theorem~\ref{main_theorem}
that the hidden constant in \eqref{eq:UBg} 
is given by the product of the hidden constant in \eqref{eq:UBd} 
and a constant that only depends on $\alpha$ and $\beta$.

We now turn to general function classes $F$
that are assumed to satisfy the following assumption. 
\hypertarget{assum}{}

\noindent\textbf{Assumption A.} \ Let $F$ be a class of
complex-valued functions on a set $D$
and let $\mu$ be a measure on $D$.
We say that $F$ and $\mu$ satisfy Assumption~A, 
if there is a metric on $F$ such that $F$ is
continuously embedded into $L_2$, separable,
and function evaluation $f\mapsto f(x)$ is, for each $x\in D$,
continuous on $F$.

Note that \hyperlink{assum}{Assumption~A} is satisfied, for example, if
\begin{itemize}
\item $F$ is a separable subset of the space of bounded functions 
equipped with the maximum distance and the measure $\mu$ is finite, \textbf{or}
\item $F$ is the unit ball of a separable normed space that is continuously embedded in $L_2$ and on which function evaluation at each point is a continuous functional, \textbf{or}
\item $F$ is a countable set of square-integrable functions, equipped with the discrete metric.
\end{itemize}
In this setting, we prove the following bound. 

\begin{theorem}
\label{thm:non-Hilbert}
Let $0<p<2$. 
There is a constant $c_p\in\N$,
depending only on $p$, 
such that for any 
$F$
and 
$\mu$ that satisfy \hyperlink{assum}{Assumption~A}
and all $m\in \N$,
 \[
 g_{c_p m}(F) \,\le\, 
 \bigg(\frac{1}{m} \sum_{k\geq m} d_k(F)^p\bigg)^{1/p} .
 \]
\end{theorem}

Theorem~\ref{thm:non-Hilbert} is an improvement over \cite{KU2}, 
where again we removed the excessive logarithmic factor.
We will also show that the result is not true for $p= 2$,
see Example~\ref{ex2}.
However, we provide a variant of Theorem~\ref{thm:non-Hilbert} under 
the weaker condition $\big((\log k)^s d_k(F)\big)\in\ell_2$ for some $s>1/2$ 
in Section~\ref{sec:boundary}.
This leads to the following corollary.

\begin{corollary}\label{cor:non-Hilbert}
Let $F$ and $\mu$ satisfy \hyperlink{assum}{Assumption~A} and 
\begin{equation}\label{eq:UBd2}
d_n(F) \,\lesssim\, n^{-\alpha} \log^{-\beta} n 
\end{equation}
for some $\alpha> 0$ and $\beta\in\R$. 
Then
\begin{equation}\label{eq:UBg2}
g_n(F) \,\lesssim\, \begin{cases} n^{-\alpha} \log^{-\beta} n
\quad &\text{if}\ \alpha>1/2,\\[2pt]
n^{-\alpha} \log^{-\beta+1} n
\quad &\text{if}\ \alpha=1/2 \text{ and } \beta>1,\\[2pt]
1 
\quad &\text{otherwise}.
\end{cases}
\end{equation}
Moreover, there exist classes $F$ such that these bounds are sharp.
\end{corollary}

Again, the hidden constant in \eqref{eq:UBg2} 
is given by the product of the hidden constant in \eqref{eq:UBd2} 
and a constant that only depends on $\alpha$ and $\beta$.
The difference compared to unit balls of RKHSs
is the case $\alpha=1/2$, 
where we need $\beta>1$ instead of $\beta>1/2$
and lose a factor $\log n$ instead of $\sqrt{\log n}$, 
see Example~\ref{ex1}. 
In addition, if $(d_k(F))\notin \ell_2$, 
then $g_n(F)$ might be bounded below by a constant, 
opposite to the RKHS setting where $g_n(F)$ tends to zero as soon as $d_k(F)$ does, see \cite{HNV}.
However, for $\alpha>1/2$, 
the results for general classes are just as strong as before.

\subsection{Remarks and related literature}
\label{biblio}

We want to add several remarks on the history of the result
and related topics.

\begin{remark}[Equivalent widths]\label{rem:Kolmogorov}
There are several quantities to measure the ``width'' of a set~$F$. 
Although we work here with the Kolmogorov numbers $d_k(F)$ as benchmark, 
let us add that these quantities coincide in $L_2$ with the 
\emph{approximation numbers} of $F$, i.e.
\[
d_k(F)\,=\,a_k(F)\,:=\,
\inf_{\substack{\ell_1,\dots,\ell_k\colon F \to \C \text{ linear}\\ \varphi_1,\dots,\varphi_k\in L_2}}\, 
\sup_{f\in F}\, 
\Big\|f - \sum_{i=1}^k \ell_i(f)\, \varphi_i\Big\|_{L_2},
\]
as the infimum in the definition of $d_k(F)$ for given 
$\varphi_1,\dots,\varphi_k$
is attained by the $L_2$-orthogonal projection onto their span, 
which is linear in any case.
The approximation numbers of a class represent
the worst-case error of an optimal linear algorithm that uses 
at most $k$ linear functionals as information. 
If $F$ is the unit ball 
of some Hilbert space~$H$, 
then the approximation numbers agree with 
the 
\emph{singular values} of the identity $\Id\colon H\to L_2$.  
In this case, the $d_k(F)$ also coincide with the \emph{Gelfand $k$-widths} 
$c_k(F)$, which represent the minimal worst-case error of 
(possibly non-linear) algorithms based on $k$ arbitrary linear functionals,
see, e.g., Chapter~4 in \cite{NW1}.
\end{remark}

\smallskip

\begin{remark}[Extreme classes $F$] 
It is interesting to note that the lower bound \eqref{eq:lower} from \cite{HNKV}
is attained already for univariate Sobolev spaces of periodic functions.
By Theorem~\ref{main_theorem}, this means that these basic classes already 
represent the most difficult RKHSs for sampling recovery 
when the numbers $d_k(F)$ are fixed.
\end{remark}

\smallskip

\begin{remark}[Least squares methods]\label{rem:alg}
The upper bounds in Theorem~\ref{main_theorem} and \ref{thm:non-Hilbert}
are proven
for a weighted least squares algorithm
using samples from a set of $cm$ points that is subsampled
from a set of $cm\log m$ i.i.d.~random points, 
see Section~\ref{sec:proof}.
Depending on the function class $F$,
the algorithm using the full set of random points may be 
constructive
but the subsampling is based on an existence result from \cite{MSS}
and is therefore not constructive.
It would be very interesting to make the subsampling constructive,
see Remark~\ref{rem:BSS}.
\end{remark}

\smallskip

\begin{remark}[Spline algorithm]\label{rem:spline}
Let $F$ be the unit ball of a RKHS $H$.
If we fix the sampling points $x_1,...,x_n$,
it is known that the smallest possible worst case error is achieved 
by the spline algorithm
\[
 S_n(f) \,:=\, \underset{g\in H \colon g(x_i)=f(x_i)}{\rm argmin}\, \Vert g \Vert_H,
\]
that is,
\[
 \inf_{\substack{\varphi_1,\dots,\varphi_n\in L_2}}\, 
\sup_{f\in F}\, 
\Big\|f - \sum_{i=1}^n f(x_i)\, \varphi_i\Big\|_{L_2}
\,=\, \sup_{f\in F}\, 
\Big\|f - S_n(f)\Big\|_{L_2},
\]
see e.g.\ \cite[Theorem~5.1]{TW}.
The function $S_n(f)$ is also known as the minimal norm interpolant
and, 
by the famous \emph{representer theorem},
can be expressed as a linear combination of the kernel functions $K(x_i,\cdot)$,
see e.g.\ \cite[Proposition~12.32]{Wain}.
Therefore, our upper bounds are true not only for the least squares algorithm,
but also for the kernel-based approximation $S_n(f)$.
Both types of algorithms are common in the context of learning,
see e.g.\ the seminal paper \cite{CS}.
\end{remark}

\smallskip

\begin{remark}[The power of i.i.d.~sampling]\label{rem:iid}
It is remarkable that, up to a logarithmic factor,
the upper bound from Theorem~\ref{main_theorem} 
is achieved with high probability for i.i.d.~random sampling points,
see \cite{KU1,U2020}.
In regard of the personal history of the authors DK and MU,
Theorem~\ref{main_theorem} is 
a byproduct 
of a series of work on 
the power of i.i.d.~sampling for approximation and integration problems
that started in \cite{HKNPU,HKNPU2}
and was also continued in \cite{HPS,KNS,KS1,KS2}.
\end{remark}

\smallskip

\begin{remark}[Expected error] 
A different approach to $L_2$-approxi\-ma\-tion is 
by using randomized algorithms
and taking the worst case expected error
instead of a worst case deterministic error.
The results in this randomized setting are quite different; 
the error of optimal algorithms does 
not depend on the tail of the sequence $(d_k(F))$.
We refer to \cite{CD,CM,K,LW,NW3,WW07}. 
\end{remark}

\smallskip

\begin{remark}[Upper bounds for infinite trace]
We note that our bounds make sense also if $d_k(F)$ is infinite 
for small $k$, but they are useless if the \emph{tail} 
of $(d_k(F))$ is not square-summable, which is the case, e.g., 
if $F$ is the unit ball of a RKHS with infinite trace, 
see~\eqref{eq:trace}. \\
An alternative approach is to bound the
numbers $g_n(F)$ by the Kolmogorov widths $d_k(F,L_\infty)$ 
in $L_\infty$: 
it is shown in \cite{T20} that there is a universal constant $c\in\N$ such that 
$g_{cm}(F) \le c \,d_m(F,L_\infty)$ for probability spaces $(D,\mathcal A, \mathcal \mu)$.
Although this bound is sometimes weaker than Theorem~\ref{thm:non-Hilbert} 
(see Example~1 in~\cite{KU2}), 
it has the great advantage that it may be applied in situations 
where the Kolmogorov widths in $L_2$ are not square-summable, see, 
e.g., \cite{TU1,TU2}. 
It would be very interesting to see whether 
it is possible to unify the two approaches.
\end{remark}

\smallskip

\begin{remark}[Tractability]
Assume now that a whole sequence of classes $F_d$ is given, 
where $d$ could be 
the dimension of the underlying domain.
For some classes we know that the curse of dimensionality is present, 
if only standard information (function values) is allowed, 
while the problem is tractable for general linear information,
see, e.g., \cite{HKNVa,NW16,V}. 
However, since the constants from Theorems~\ref{main_theorem}
and \ref{thm:non-Hilbert} are independent of the dimension,
it is possible to transfer certain tractability properties
from linear information to standard information,
see, e.g., \cite{KUV,KSUW,NW3}.
\end{remark}

\smallskip

\begin{remark}[Separability of $F$] 
\label{rem:assumptions}
Contrarily to the $\ell_2$-summability of the Kolmogorov widths, 
it should be possible
to remove the
separability assumption on the 
class $F$, at least
in Theorem~\ref{main_theorem}, 
by adding a term ${{\rm tr}_0(K)}/{m}$ inside the square root
in the right-hand side, 
as done in \cite{MU}. 
\end{remark}

\smallskip

\begin{remark}[Discretization of continuous frames]\label{rem:frames}
A related problem is the question whether a \emph{continuous frame} 
for a Hilbert space may be sampled to obtain a frame, 
see~\cite{C03} for details. 
This problem, which was originally posed in the physics book~\cite{AAG}, 
has only recently been solved in~\cite{FS}, 
see also the survey~\cite{B18}. Although seemingly independent, this line of research
uses remarkably similar methods.
We leave it to future research to better understand and expand the connections.
\end{remark}

\subsection{Outline}

The rest of the paper can be outlined as follows.
Sections~\ref{RKHS_setting}--\ref{sec:proof} 
form the proof of Theorem~\ref{main_theorem}.
In Section~\ref{RKHS_setting},
we collect some basics on the RKHS setting.
In Section~\ref{section_concentration}, 
we obtain our initial sample of $\mathcal O (m\log m )$ points
based on a concentration inequality for infinite matrices.
The subsampling is performed in
Section~\ref{section_subsampling}, which applies the solution to the Kadison-Singer 
problem in a slightly original way, leading to the core of the proof 
in Section~\ref{sec:proof}. 
In Section~\ref{sec:proof-general},
we prove our results for general function classes
by constructing a suitable RKHS, 
on which a local version of Theorem~\ref{main_theorem} (Theorem~\ref{main_theorem_local})
can be applied. 
Finally, in Section~\ref{sec:ex},
we present examples,
applying our result to tensor product problems and
showing that our upper bounds are sharp. 

{\bf Acknowledgement:} The authors thank Albert Cohen, 
Daniel Freeman, 
Aicke Hinrichs and Erich Novak  
for useful discussions and detailed corrections.
Moreover, DK and MU want to acknowledge that MD obtained 
Theorem~\ref{main_theorem} on his own and presented it at the
\emph{MASCOT-NUM Workshop on ``Optimal Sampling for Approximation''} 
at the IHP in Paris on March 10, 2022. 
The current work is the product of subsequent collaboration. 

\goodbreak

\section{Hilbert space setting}
\label{RKHS_setting} 

We first consider the case where 
$F$ is the unit ball of a separable Hilbert space $H$
with reproducing kernel $K\in\C^{D\times D}$.
We refer to \cite{MU} and references therein for theoretical background on RKHSs.

Thanks to the finite trace assumption \eqref{eq:trace}, 
we know that the 
identity map
$\Id\co H\to L_2$ is Hilbert-Schmidt, 
thus its left and right singular 
vectors $(b_k)_{k\in \II}$ and $(\sigma_k b_k)_{k \in \II}$ 
are orthonormal families in $L_2$ and $H$, respectively.
Here, we only list the singular vectors with respect to the nonzero 
singular values $\sigma_k>0$,
and the index set is of the form $\II=\{k\in\N_0 \colon k< M\}$ 
with $M\in\N\cup\{\infty\}$. 
The singular vectors satisfy 
\[
 \langle f,b_k\rangle_{L_2} = \langle f,\sigma_k^2 b_k\rangle_{H}
 \quad \text{for all } f\in H \text{ and } k\in \II.
\]
We use the convention that $\N_0:=\{0,1,2,\dots\}$
and the singular values are arranged in a non-increasing order.
In particular, $\sum_{k\in\II}\sigma_k^2 <\infty$ and 
the Kolmogorov width $d_m(F)=\sigma_m$ is attained 
by the $L_2$-orthogonal projection $P_m$ onto $V_m=\Span\{b_k\co k<m\}$.
Moreover, the separability of $H$ ensures that the equality
\[
K(x,y)=\sum_{k\in\II}\sigma_k^2\,{b_k(x)}\overline{b_k(y)}
\]
holds for 
all $x,y\in D_0$ with some set $D_0\subset D$ satisfying $\mu(D\setminus D_0)=0$.
We therefore have the identity
\begin{equation}\label{eq:Fourier-series}
 f(x) \,=\, \sum_{k\in\II} \scalar{f}{b_k}_{L_2} b_k(x)
 \quad\text{for all }f\in H \text{ and } x\in D_0.
\end{equation}
Our sampling points will be contained in the set $D_0$.

As a consequence of the following lemma, we only have to show the validity of 
Theorem~\ref{main_theorem} for all $1\leq m<M$.

\begin{lemma}
Let $M=\min\{m\in\N\colon d_m(F)=0\}<\infty$. Then we have $g_n(F)=0$ for all $n\ge M$.
\end{lemma}

\begin{proof}
For $x\in D_0$, we write $b(x)=(b_0(x),\hdots,b_{M-1}(x))$.
Then there are points $x_0,\hdots,x_{M-1}\in D_0$ such that
every $b(x)$ 
is contained in the span of the vectors $b(x_i)$.
We write $b(x)=\sum \varphi_i(x) b(x_i)$ with coefficients $\varphi_i(x) \in \C$.
By \eqref{eq:Fourier-series},
we have
\[
 f(x) 
 = \sum_{k<M} \langle f,b_k\rangle_{L_2} \sum_{i<M} \varphi_i(x) b_k(x_i)
 = \sum_{i<M} f(x_i) \varphi_i(x),
\]
for all $x\in D_0$ and $f \in H$.
Thus, the identity $f=\sum f(x_i) \varphi_i$
holds almost everywhere. 
Moreover, the functions $b_0,\dots,b_{M-1}$ restricted to $D_0$
form a basis of $\Span\{\varphi_i \colon i<M\}$,
and thus $\varphi_i \in L_2$.
\end{proof}

We fix an integer $1\le m< M$ for the rest of the proof
of Theorem~\ref{main_theorem}.

\goodbreak

\section{Concentration inequality}
\label{section_concentration}

As proposed in \cite{KU1} and 
applied in \cite{KUV,KU2,MU,NSU, U2020}, 
we define the probability density
\[
\rho_m(x)=\frac{1}{2}\left(\frac{1}{m}\sum_{k<m} |b_k(x)|^2+\frac{\sum_{k\geq m}\sigma_k^2\,|b_k(x)|^2}{\sum_{k\geq m}\sigma_k^2}\right).
\]
and
draw i.i.d.~random points $x_1,\dots,x_n\in D$ according to this density. 
We define the $M$-dimensional vectors $y_1,\dots,y_n$ by
\[
(y_i)_k=\left\{\begin{array}{cl} \rho_m(x_i)^{-1/2}\, b_k(x_i)&\text{ if }0 \le  k<m,\\ 
\rho_m(x_i)^{-1/2}\,\gamma_m^{-1} \,\sigma_k\, b_{k}(x_i)&\text{ if }m\leq k<M,
\end{array}\right.
\]
where
\[
\gamma_m \,:=\, 
\max\bigg\{\sigma_m, \, \sqrt{\frac{1}{m}\sum_{k\geq m}\sigma_k^2} \, \bigg\}
>0.
\]
Note that $\rho_m(x_i)>0$ almost surely.
It follows from these definitions that $y_i\in \ell_2(\II)$ with
\[
\|y_i\|_{2}^2
=\rho_m(x_i)^{-1}\bigg(\sum_{k<m}|b_k(x_i)|^2+\gamma_m^{-2}\sum_{k\geq m}\sigma_k^2\,|b_k(x_i)|^2\bigg) \leq 2m, 
\]
and 
\[
\E(y_i y_i^*)=\diag(1,\dots,1,\sigma_m^2/\gamma_m^2,\sigma_{m+1}^2/\gamma_m^2,\dots)
=:E,
\]
with $\|E\|_{2\to 2}=1$ since $\sigma_k^2/\gamma_m^2\leq 1$ for $k\geq m$.
Here, $\diag(v)$ denotes a diagonal matrix with diagonal $v$, 
and $\|\cdot\|_{2\to 2}$ denotes the spectral norm of a matrix.

We apply the following concentration inequality for infinite matrices, 
which was proved by Mendelson and Pajor in \cite[Theorem 2.1]{MP}. 
We use a version of this result from
\cite[Theorem~1.1]{MU} and \cite[Theorem~5.3]{NSU}.

\begin{lemma}
\label{lemma_concentration}
Let $n\geq3$ and $y_1,\dots,y_n$ be i.i.d.\ random sequences 
from $\ell_2(\II)$ 
satisfying $\|y_i\|_{2}^2\leq 2m$ almost surely and $\|E\|_{2\to 2}\leq1$, with $E=\E(y_i y_i^*)$.
Then, for $0\le t\le1$, 
\[
\P\(\Big\|\frac{1}{n}\sum_{i=1}^n y_i y_i^*-E\Big\|_{2\to 2}>t\) \,\leq\, 2^{3/4}\,n\,\exp\(-\frac{nt^2}{42\, m}\).
\]
\end{lemma}

For $t=1/2$, this probability is less than $1/2$ as soon as $\frac{n}{\log(4n)}\geq 168 \,m$. In the sequel we take
\[
n= \lfloor C_0m\log(m+1)\rfloor, 
\]
with $C_0$ large enough, so that the previous inequality holds true. 
(One can take $C_0=10^4$, for instance.) 
Thanks to Lemma \ref{lemma_concentration}, we know that there exists 
a realization $x_1,\dots,x_n\in D_0$ of the random sampling such that the corresponding 
family $y_1,\dots,y_n$ satisfies
\begin{equation}
\label{eq_concentration}
\Big\|\frac{1}{n}\sum_{i=1}^n y_i y_i^*-E\Big\|_{2\to 2}\leq \frac{1}{2}.
\end{equation}
We fix such a sequence for the rest of the proof of Theorem \ref{main_theorem}.

\goodbreak

\section{Subsampling of infinite vectors}
\label{section_subsampling}

We now want to apply the solution to the 
Kadison-Singer problem, 
or specifically to 
Weaver's conjecture, 
to the sum of rank-one 
matrices
\[
\frac 1 n \sum_{i=1}^n y_i y_i^*,
\]
in order to find a subsampling of order $m$ preserving the 
spectral properties of the sum. 
The original result comes from the celebrated paper \cite{MSS} 
by Marcus, Spielman and Srivastava, and has already been applied 
numerous times in approximation theory, see for instance 
\cite{CD,K,KU1,KU2,MU,NSU,NOU2016,T20}.
However, the original subsampling strategy only works for 
finite matrices. 
The main result of this section is the following infinite-dimensional 
variant, that might be of independent interest.

\begin{prop}
\label{prop:subsampling}
There are absolute constants $c_1\leq 43200$, $c_2\geq 50$, $c_3\leq 21600$, 
with the following properties. 
Let $n,m\in\N$ and $y_1,\dots,y_n$ be vectors from $\ell_2(\N_0)$ 
satisfying $\|y_i\|_{2}^2\leq 2m$ and 
\begin{equation}\label{eq:prop}
\norm{\frac{1}{n}\sum_{i=1}^n y_i y_i^* 
- \begin{pmatrix} I_m &0\\0&\Lambda\end{pmatrix}}_{2\to2}
\le \frac12,
\end{equation}
for some Hermitian matrix $\Lambda$ with $\|\Lambda\|_{2\to 2}\leq1$, 
where $I_m\in\C^{m\times m}$ denotes the identity.\\
Then, there is a subset $J\subset\{1,\dots,n\}$ 
with $|J|\le c_1 m$, such that
\[
\bigg(\frac{1}{m}\sum_{i\in J} y_i y_i^* \bigg)_{< m}
\,\ge\, c_2\,I_m
\qquad\text{and}\qquad
\frac{1}{m}\sum_{i\in J} y_i y_i^* 
\,\leq\, c_3\,I,
\]
where $A_{<m}:=(A_{k,l})_{k,l<m}$ and $A\le B$ denotes the Loewner order of Hermitian matrices $A$ and $B$.
\end{prop}

The conclusion can be understood as an upper bound on the largest eigenvalue of $A= \sum_{i\in J} y_i y_i^*$
and a lower bound on the smallest eigenvalue of $A_{<m}$. 
Note that the constants in Proposition~\ref{prop:subsampling}, and hence also the final sampling size, 
are independent of $n$, the original sampling size. 
The rest of this section is devoted to the proof of this proposition. 

\subsection{Reduction to finite dimension}
\label{subsec:dimension}

Let $U_0$ be a matrix whose columns form an orthonormal basis of 
\[
 {\rm span}\left\{ (y_i)_{\geq m} \colon i=1,\hdots, n \right\} \subset \ell_2, 
\]
where $(y_i)_{\ge m}=((y_i)_k)_{k\ge m}$.
Clearly, $U_0$ has at most $n$ columns. 
Then we have that $U_0^*U_0$ is the identity matrix 
and in particular the spectral norm of $U_0$ and $U_0^*$ equals one.
We set
\[
U=\begin{pmatrix} I_m &0\\0&U_0\end{pmatrix}, 
\]
which is a matrix that satisfies $U^*U=I_p$, 
where $p\le m+n$,
and therefore also $U$ and $U^*$ have unit norm.
We choose vectors $z_i\in\C^p$ that satisfy $Uz_i=y_i$ for all $i\le n$.  
Such vectors exist since $y_i$ is contained in the span of the columns of $U$.  
Then we also have $z_i=U^*Uz_i = U^* y_i$. 

Let $E = \begin{pmatrix} I_m &0\\0&\Lambda\end{pmatrix}$  be the matrix from Proposition~\ref{prop:subsampling}.
We define 
\[
\hat E=U^*\,E\,U=\begin{pmatrix} I_m &0\\0& E'\end{pmatrix} 
\quad \text{ where } \quad 
\|E'\|_{2\to 2}\le \|E\|_{2\to 2}
\leq 1.
\]
With the norm bounds on $U$ and $U^*$, 
equation~\eqref{eq:prop} gives
\[
\Big\|\frac{1}{n}\sum_{i=1}^n z_i z_i^*-\hat E\Big\|_{2\to 2}
=\Big\|U^*\(\frac{1}{n}\sum_{i=1}^n y_i y_i^* -E\)U\Big\|_{2\to 2}\leq \frac{1}{2}. 
\]

\subsection{Approximating the identity}
 
In addition to finite dimension, 
the result from \cite{MSS} requires the matrix $\frac 1n \sum_{i=1}^n z_iz_i^*$ 
to be close to the identity 
in spectral norm, and this is not ensured here. 
To mitigate this defect, we artificially add 
rank-one matrices $z_i z_i^*\in \C^{p\times p}$  
for $i=n+1,\dots,q$ 
in the following way.

As $I_p-\hat E$ is positive semi-definite, 
we can decompose it as a sum of rank-one matrices 
\[
I_p-\hat E=\begin{pmatrix}0&0\\0&I_{p-m}-E'\end{pmatrix}
=\sum_{j=1}^{p-m}t_j t_j^*,
\]
where $t_j\in \C^{p}$. 
We now choose
\[
z_i=\sqrt{\frac{ n}{n_{j(i)}}}\,t_{j(i)},\quad n_j
=\left\lceil\frac{n}{2m}\,\|t_{j}\|_2^2\right\rceil,
\]
with $j(i)\in\{1,\dots,p-m\}$ such that $\{z_{i}, \,i=n+1,\dots, q\}$ 
contains exactly $n_j$ copies of each $\sqrt{n/n_j}\,t_j$. 
In this way, for $i>n$, the first $m$ entries of $z_i$ are zero 
since this is true for the $t_j$,
\[
\|z_i\|_2^2 \leq \frac{n}{n_{j(i)}} \,\|t_{j(i)}\|_2^2\leq 2m, 
\]
and
\begin{equation*}
\begin{split}
\Big\|\frac{1}{n}\sum_{i=1}^qz_i z_i^*-I_p\Big\|_{2\to2}
&=\Big\|\frac{1}{n}\sum_{i=1}^nz_i z_i^*+\sum_{j=1}^{p-m}t_j t_j^*-I_p\Big\|_{2\to2} \\
&=\Big\|\frac{1}{n}\sum_{i=1}^nz_iz_i^*-\hat E\Big\|_{2\to2}
\leq \frac{1}{2}.
\end{split}
\end{equation*}

\begin{remark}
As $\|t_j\|_2^2\leq \|I_{p}-\hat E\|_{2\to 2}\leq1$, we count
\[
q=n+\sum_{j=1}^{p-m}n_j\leq n+\sum_{j=1}^{p-m}\(1+\frac{n}{2m}\)\leq 
n+(p-m)\frac{n}{m}=\frac{np}{m}.
\]
Conversely, taking traces in $\C^{p\times p}$, we find
\[
\frac p 2=\Tr\(\frac{1}{2} I_p\)\leq \Tr\(\frac{1}{n}\sum_{i=1}^qz_i z_i^*\)
=\frac{1}{n}\sum_{i=1}^q\|z_i\|_2^2\leq \frac{2m q}{n}.
\]
So, we obtain $n/m\geq q/p\ge n/4m$. Recall that, given $m$ the dimension of the approximation space $V_m$, 
we took $n=\mc O(m\log m)$ initial sample points, 
and vectors $z_i$ of size $p=\mc O(m \log m)$. 
Hence, the number of such vectors is $q=\mc O(m \log^2 m)$.
Surprisingly, we do not use estimates on $p$ and $q$ in the rest of the argument.
\end{remark}

\begin{remark}
In fact, we did not need an exponential 
speed of convergence in the concentration inequality. 
The reduction of the sample size to $\mathcal O(m)$ points
works for any initial set of sampling points satisfying \eqref{eq_concentration}.
If the cardinality of the initial sample is
$n=m\, \ell(m)$, where $\ell(m)$ is any positive function of $m$, we get $p=\mc O(m\, \ell(m))$ 
and $q=\mc O(m \,\ell(m)^2)$. \\
\end{remark}

\subsection{Reduction of the sample size}
\label{subsec:KS}
We can now use the Kadison-Singer solution from \cite{MSS} in an iterated way, 
as proposed in Lemma~3 of \cite{NOU2016}, 
and later used in \cite{ CD,KU1,KU2, LT, MU,NSU, T20}. 
The following lemma is obtained from Corollary~B and Lemma~1 in \cite{NOU2016}.

\begin{lemma}
\label{lemma_subsample}
Let $z_1,\dots,z_q\in\C^p$ with $\|z_i\|_{2}^2\le \delta$ and
\[
\alpha I_p \leq \sum_{i=1}^q z_i z_i^* \leq \beta I_p
\]
for some $\beta \ge \alpha>100\,\delta>0$.
Then there is a partition of $\{1,\hdots,q\}$ into sets $J_1,\hdots,J_t$ such that, for all $s\le t$,
we have
\[
25\, \delta I_p \,\le\, \sum_{i\in J_s} z_i z_i^* \,\le\, 3600\,\frac{\beta}{\alpha}\, \delta I_p.
\]
\end{lemma}

\begin{proof}
Since the matrix $M=\sum_{i=1}^q z_i z_i^*$ is positive, we may define $\tilde z_i = M^{-1/2} z_i$.
Then we have $\sum_{i=1}^q \tilde z_i \tilde z_i^* = I_p$
and $\Vert \tilde z_i \Vert_2^2 \le \delta/\alpha =: \delta^\prime < 1/100$.
By Corollary~B and Lemma~1 in \cite{NOU2016},
noting that the constant $C$ from Lemma~1 is at most 36,
we get a partition of $\{1,\hdots,q\}$ into sets $J_1,\hdots,J_t$ such that, for all $s\le t$,
we have
\[
25\, \delta^\prime\, I_p \,\le\, \sum_{i\in J_s} \tilde z_i \tilde z_i^* \,\le\, 3600\, \delta^\prime\, I_p.
\]
Now, using 
\[ \sum_{i\in J_s} z_i z_i^* = M^{1/2} \,\sum_{i\in J_s} \tilde z_i \tilde z_i^*\, M^{1/2},\]
we get the statement.
\end{proof}

Note that one could obtain 
better constants
by adapting 
the proof of Theorem 2.3 from \cite{NSU}.
In our case, we have $\delta=2m$, $\alpha=n/2$ and $\beta=3n/2$.
The relation $\alpha>100\,\delta$ is satisfied.
We thus obtain
\[
50\,m I_p \,\le\, \sum_{i\in J_s} z_i z_i^* \,\le\, 21 600\, m I_p.
\]
for every $J_s$ from the partition.
Moreover, the inequality
\[
 \frac{n}{2} I_p \,\le\, \sum_{i=1}^q z_i z_i^*
 \,=\, \sum_{s=1}^t \sum_{i\in J_s} z_i z_i^*
 \,\le\, 21 600\, t m I_p
\]
implies that one of the sets $J'=J_s$ from the partition must satisfy 
\[
  | J' \cap \{1,\dots,n\} | \,\le\, \frac{n}{t}\,\le\, 43200\,m.
\]

After applying Lemma \ref{lemma_subsample} and removing the indices 
from $J'\cap\{n+1,\dots,q\}$ corresponding to artificially added vectors, 
we are left with a set $J:=J'\cap\{1,\dots,n\}$ of cardinality
\[
|J| \,\le\, 43200\, m.
\]

It remains to show that the artificial vectors do not 
interfere with our desired properties. 
For this, recall that $(z_i)_k=(y_i)_k$ for $k<m$ and $i\leq n$, whereas 
the first $m$ entries of $z_i\in\C^p$ are zero for $i>n$. 
Hence,
\[
\(\sum_{i\in J}y_i y_i^*\)_{<m}
=\(\sum_{i\in J'}z_i z_i^*\)_{<m}
\,\ge\, 50 \,m\, I_m,
\]
where we use a simple linear algebra fact on self-adjoint matrices $A$:
\[
\lambda_{\min}(A_{<m})
=\inf_{\substack{z\in \C^p, \,\|z\|_2=1\\z_k=0\text{ for }k\geq m }}z^*Az 
\geq \inf_{z\in \C^p,\, \|z\|_2=1}z^*Az = \lambda_{\min}(A).
\]

Similarly, and using positive definiteness, we have
\[
\sum_{i\in J}z_i z_i^*
\leq \sum_{i\in J'}z_i z_i^*
\leq 21600\,m\,I_{p}.
\]
With the orthogonal transformation $U$ 
from Section~\ref{subsec:dimension}, we get
\[
\Big\|\sum_{i\in J} y_i y_i^*\Big\|_{2\to 2}
=\Big\|U\(\sum_{i\in J} z_i z_i^* \)U^*\Big\|_{2\to 2} 
\le \Big\|\sum_{i\in J}z_i z_i^*\Big\|_{2\to 2}
\leq 21600 \,m.
\]
This proves Proposition~\ref{prop:subsampling}. \\ \qed

\begin{remark} \label{rem:BSS}
It would be an interesting improvement to use
the result of Batson, Spielman and Srivastava, see~\cite{BSS}, 
instead of \cite{MSS} for the subsampling. 
This earlier paper is applied to approximation theory 
in e.g.\ \cite{LT,NOU2011,T18} and more recently in \cite{BSU}. 
It presents a slightly less powerful method, requiring additional weights, 
but comes with an almost linear algorithmic complexity, see \cite{LS}, 
and much smaller constants, which could make the bound presented here sharp 
also in terms of numerical values.
\end{remark}

\begin{remark}
We recently learned that it might be possible to use results from~\cite{FS}, 
which work directly in an infinite-dimensional setting, 
to save the reduction to a finite dimension in Section~\ref{subsec:dimension}. 
However, as the core of our method is~\cite{MSS}, we decided 
to keep our more direct deduction.
\end{remark}

\section{Proof of the main theorem}
\label{sec:proof}

We now have all the tools for proving Theorem \ref{main_theorem}. 

To obtain our sampling points,
we combine~\eqref{eq_concentration} for our initial vectors $y_i\in\ell_2(\II)$ 
with Proposition~\ref{prop:subsampling}.
Clearly, Proposition~\ref{prop:subsampling} stays true
if we replace $\N_0$ by the possibly finite index set $\II$. 
We obtain
points $x_1,\dots,x_n \in D_0$ with $n\le 43200\, m$ 
such that the vectors
\[
(y_i)_k=\left\{\begin{array}{cl} \rho_m(x_i)^{-1/2}\,b_k(x_i)&\text{ if } 0\le k<m,\\ 
\rho_m(x_i)^{-1/2}\,\gamma_m^{-1} \,\sigma_k\, b_{k}(x_i)&\text{ if }m\leq k<M,\end{array}\right.
\]
satisfy
\[
 \(\sum_{i=1}^n y_iy_i^*\)_{<m} \,\ge\,50\,m\,I,
\]
and 
\[
\(\sum_{i=1}^n y_iy_i^*\)_{\ge m} 
\,\le\, 21600\,m\,I,
\]
where we use the notation $A_{\ge m}=(A_{k,l})_{k,l\ge m}$
for a matrix $A$.

As in earlier papers, we use the \emph{weighted least squares estimator}
\[
A_{n}(f) \,:=\, 
\underset{g\in V_m}{\rm argmin}\, \sum_{i=1}^n \frac{\vert g(x_i) - f(x_i) \vert^2}{\rho_m(x_i)}
\]
with $V_m$ and $\rho_m$ as defined in Sections~\ref{RKHS_setting} 
and \ref{section_concentration}, respectively, see~\cite{KU1}.
This algorithm 
may be written as 
\[
A_{n}(f) \,=\, \sum_{k=1}^m (G^+ N f)_k\, b_k 
\]
where $N\colon F \to \C^n$ with 
$N(f):=\left(\rho_m(x_i)^{-1/2}f(x_i)\right)_{i\leq n}$ 
is the \emph{information mapping} and
$G^+\in \C^{m\times n}$ is the Moore-Penrose inverse of the matrix 
\[
 G := \left(\rho_m(x_i)^{-1/2}\,b_k(x_i)\right)_{i\le n,\, k\le m} 
\in \C^{n\times m}.
\]
Since we have the identity $\overline{G^*G}=(\sum_{i=1}^n y_iy_i^*)_{<m}$,
the matrix $G$ has full rank 
and the spectral norm of $G^+$ is bounded by $(50 m)^{-1/2}$.
In particular,
the argmin in the definition of $A_{n}$ is uniquely defined
and $A_n$ satisfies $A_{n}(f)=f$ for all $f\in V_m$. \\
Denoting with $Q_m$ the $L_2$-orthogonal projection onto
${\rm span}\{b_k \colon k\ge m\}$,
we obtain for any $f\in H$ that
\begin{align*}
\allowdisplaybreaks
\norm{f-A_{n}(f)}_{L_2}^2 
\,&=\,  \norm{f-P_mf}_{L_2}^2 + \norm{P_m  f - A_{n}(f)}_{L_2}^2 \\
\,&=\,  \norm{Q_mf}_{L_2}^2 + \norm{A_{n}(f- P_m f)}_{L_2}^2 \\
\,&=\, \norm{Q_mf}_{L_2}^2  + \norm{G^+ N (f- P_m f)}_{\ell_2^n}^2 \\
 &\le\, \sigma_m^2\,\norm{Q_mf}_{H}^2 +\norm{G^+}_{2\to2}^2 
\cdot \norm{N(f-P_m f)}_{\ell_2^m}^2.
\end{align*}
By \eqref{eq:Fourier-series} we have $N(f-P_mf)=\Phi \xi_f$, where
\[
 \Phi=\(\rho_m(x_i)^{-1/2}\sigma_k\, b_k(x_i)\)_{ i\le n,\,k\geq m}
 \quad\text{and} \quad \xi_f = (\langle f,\sigma_k \, b_k\rangle_H)_{k\geq m}.
\] 
The matrix $\Phi$ satisfies 
\[
\overline{{\Phi_{}^{*}\Phi}}  
\,=\, \gamma_m^2 \(\sum_{i=1}^n y_iy_i^*\)_{\ge m}
\]
and therefore its spectral norm is bounded by 
$\(21600\,m \,\gamma_m^2\)^{1/2}$.
Thus,
\[
 \norm{N(f-P_m f)}_{\ell_2^m}^2 \,\le\, 21600\,m \,\gamma_m^2 \Vert \xi_f\Vert_2^2
 \,=\, 21600\,m \,\gamma_m^2\, \Vert Q_m f\Vert_H^2.
\]
In summary, we obtain for all $1\le m <M$ the bound
\begin{equation}\label{eq:error-A}
\norm{f-A_{n}(f)}_{L_2}^2
\,\leq\, 433 \,
\max\bigg\{\sigma_m^2, \, \frac{1}{m}\sum_{k\geq m}\sigma_k^2 \, \bigg\}
\, \norm{Q_m f}_{H}^2.
\end{equation}
for all $f\in H$ and some $n\le 43200 \,m$.
Taking the supremum over $f\in F$ and using that
\[
\max\bigg\{\sigma_m^2, \, \frac{1}{m}\sum_{k\geq m}\sigma_k^2 \, \bigg\}
\,\le\,\frac{2}{m}\sum_{k\geq\lceil m/2\rceil}\sigma_k^2,
\]
we obtain 
\[
g_{43200\,m}(F)^2 \,\le\, \frac{866}{m}\sum_{k\geq \lceil m/2\rceil} \sigma_k^2. 
\]
This finishes the proof of Theorem~\ref{main_theorem} with 
$c=43200\cdot 866 $. \\ \qed

In fact, equation \eqref{eq:error-A} provides
a local upper bound which is sometimes superior to Theorem~\ref{main_theorem}.
We therefore state it separately.

\begin{theorem}
\label{main_theorem_local}
Let $\mu$ be a measure on a set $D$ 
and let $F\subset L_2(\mu)$ be the unit ball of a separable RKHS $H$ 
such that the finite trace assumption~\eqref{eq:trace} holds.
For $m\in \N$,
let $P_m$ be the orthogonal projection onto the span $V_m$
of the singular vectors corresponding to the $m$ largest singular values
of the embedding of $H$ into $L_2$.
Then there exist $x_1,\hdots,x_n \in D$
and $\varphi_1,\hdots,\varphi_n \in V_m$, where $n \le 43200 m$,
such that, for all $f\in H$,
\begin{equation*}
\Big\|f - \sum_{i=1}^n f(x_i)\, \varphi_i\Big\|_{L_2}^2
\leq 433 
\max\bigg\{d_m(F)^2, \, \frac{1}{m}\sum_{k\geq m}d_k(F)^2 \bigg\}
 \big\Vert f - P_m f\big\Vert_{H}^2.
\end{equation*}
\end{theorem}

\medskip

\begin{remark}
For the purpose of Theorem~\ref{main_theorem} it  
was enough
to bound 
$\norm{f-P_mf}_H \le \norm{f}_H$.
However, Theorem~\ref{main_theorem_local}
will be of advantage later for the study of 
general classes
since it is able to see additional decay 
of the Fourier coefficients $\langle f, b_k\rangle_{L_2}$
compared to the decay implied by $f\in H$.
Note that faster decay of the Fourier coefficients often corresponds 
to higher smoothness of the 
function.
In a certain sense, this means that the algorithm 
is universal.
The error has the optimal rate of decay for any smoothness higher 
than the smoothness of $H$.
\end{remark}

\begin{remark}
The condition on the point sets can also be given 
by finite matrices that are related to the kernel $K$ of the Hilbert space.
For this, 
let us define 
$K_m(x,y):=\sum_{k< m} b_k(x) b_k(y)$, 
and 
$R_m(x,y):=\sum_{k\ge m} \sigma_k^2\, b_k(x) b_k(y)$.  
The non-zero singular values of 
$GG^*$ 
are the same as those of $G^*G$, 
and the non-zero singular values of 
$\Phi\Phi^*$ 
are the same as those of~$\Phi^*\Phi$,
where $G$ and $\Phi$ are from above. 
Hence, 
the algorithm $A_n$ based on points $x_1,\dots,x_n$ satisfies 
the error bound above (up to a constant) if
\[
c\,m \,\le\, \lambda_m\(GG^*\) 
\,=\, \lambda_m\(\(\frac{K_m(x_i,x_j)}{\sqrt{\rho_m(x_i)\rho_m(x_j)}} \)_{i,j=1}^n\)
\]
and
\[
\(\frac{R_m(x_i,x_j)}{\sqrt{\rho_m(x_i)\rho_m(x_j)}} \)_{i,j=1}^n \,=\,
\Phi\Phi^* \,\le\, C \, m\,\gamma_m^2\,I
\]
for some constants $c,C>0$, where 
$\lambda_m$ denotes the $m$th eigenvalue. 
It would be interesting to find a property that only involves 
the kernel $K$ directly (instead of the truncated kernels above), 
or to verify that a similar property  
characterizes \emph{good} point sets, 
in a way similar to Proposition~1 of~\cite{HKNVa} for 
integration. 
\end{remark}

\subsection{Proof of Corollary~\ref{cor:Hilbert}}
\label{sec:polynomial}

For the given bounds on the sampling numbers for sequences of 
polynomial decay, we only need to note that
\[
 \frac1n \sum_{k\ge n} k^{-a} \log^{-b} k \,\lesssim\, \begin{cases} n^{-a} \log^{-b} n
  \quad &\text{if}\ a>1, \, b\in\R,\\[2pt]
 n^{-a} \log^{-b+1} n
\quad &\text{if}\ a=1,\, b>1.
\end{cases}
\]
Hence, Corollary~\ref{cor:Hilbert} immediately follows from Theorem~\ref{main_theorem},
and the existence of $F$ where the bounds are attained comes from \eqref{eq:lower}, see \cite{HNKV}.\\ \qed

\section{General function classes}
\label{sec:proof-general}

We now prove all results related to general function classes.

\subsection{Proof of Theorem~\ref{thm:non-Hilbert}}

We will make use of 
the following observation from \cite[Lemma 3]{KU2}.
We copy its proof 
for completeness.

\begin{lemma}\label{lem:basis}
Let $F \subset L_2$ and let $L_2$ be infinite-dimensional.
There is an orthonormal system 
$(b_k)_{k\in \N_0}$
in $L_2$ such that for all $m\ge 1$,
the orthogonal projection $P_m$ onto
$V_m= \Span\{b_k\co k<m\}$ 
satisfies
\begin{equation}
\label{eq:projections}
 \sup_{f\in F} \Vert f - P_m f \Vert_{L_2}  
\,\le\, 2\, d_{\lfloor m/4\rfloor}(F).
\end{equation}
\end{lemma}

\begin{proof}
Clearly it is enough to find an increasing sequence of 
subspaces of $L_2$,
 \[
  U_1 \subseteq U_2 \subseteq U_3 \subseteq \hdots,
  \qquad
  \dim(U_m) \le m,
 \]
 such that the projection $P_m$ onto $U_m$ satisfies \eqref{eq:projections}.
 By the definition of $d_k(F)$, $k\in\N_0$, 
 there is a subspace $W_k\subset L_2$ of dimension~$k$
and a mapping $T_k\colon F \to W_k$
 such that 
 \[
  \sup_{f\in F} \Vert f - T_k f \Vert_{L_2} \,\le\,  2\, d_k(F).
 \]
 This is also true if $d_k(F)=0$.
 We let $U_m$ be the space that is spanned by the union of the 
 spaces $W_{2^\ell}$ over all $\ell\in \N_0$ such that $2^\ell\le m/2$. 
 Note that $U_m$ contains a subspace 
 $W_k$ with $k\ge \lfloor m/4\rfloor$.
 Therefore, 
 $P_m f$ is at least as close to $f$ as $T_k f$ for some $k\ge \lfloor m/4\rfloor$,
 which implies~\eqref{eq:projections}.
\end{proof}

We now turn to the proof of Theorem~\ref{thm:non-Hilbert}.
The basic idea is to construct a suitable reproducing kernel Hilbert space $H$
that contains a dense subset of $F$
and apply Theorem~\ref{main_theorem_local} to this Hilbert space.
It will be important to use the local bound from Theorem~\ref{main_theorem_local} 
instead of the global bound from Theorem~\ref{main_theorem}.

\begin{proof}[Proof of Theorem~\ref{thm:non-Hilbert}]

Without loss of generality, we assume that $L_2$ is infinite-dimensional.
Moreover, we assume that $d_k(F)$ is finite for $k\ge k_0$ and that $(d_k(F))_{k\ge k_0}\in\ell_p$. 
Otherwise, the statement is trivial.

By Lemma~\ref{lem:basis},
there is an orthonormal system 
$(b_k)_{k\in \N_0}$ such that \eqref{eq:projections} is satisfied for all $m\in\N$.
We will 
consider $b_k$ as a function,
where we fix an arbitrary representer from the equivalence class in $L_2$.
We call $$\hat f(k) := \scalar{f}{b_k}_{L_2}$$
the $k$th Fourier coefficient of $f$.
Moreover, we fix a countable dense subset $F_0$ of $F$
and set $\sigma_k=\max\{1,k\}^{-\alpha}$ for all $k\in \N_0$
and some $\alpha\in(1/2,1/p)$. 
Then we have $(\sigma_k)\in\ell_2$.

We now want to define a RKHS on a set $D_0 \subset D$,
with $\mu(D\setminus D_0)=0$, which admits the orthonormal basis $(\sigma_k b_k)$
and contains the set $F_0$.
Such a Hilbert space will have the reproducing kernel 
\[
 K(x,y) \,=\, \sum_{k\in\N_0} \sigma_k^2 b_k(x) \overline{b_k(y)}.
\]
To find a suitable set $D_0$, we first note that 
\begin{equation}\label{eq:traceofH}
 \int_D K(x,x)\,d\mu(x) \,=\, \sum_{k\in\N_0} \sigma_k^2 < \infty
\end{equation}
and thus $K(x,x)$ is finite for all $x\in D\setminus E$
with a null set $E\subset D$.
Moreover, for all $f\in F_0$, we have
\[
 \sum_{k\ge 1} k\, |\hat f(k)|^2 
 \,=\, \sum_{n\ge 0} \sum_{k>n} |\hat f(k)|^2
 \,=\, \sum_{n\ge 0} \Vert f - P_n f \Vert_{L_2}^2 < \infty,
\]
where we use \eqref{eq:projections} and the assumptions on $F$.
The Rademacher-Menchov Theorem, see e.g.~\cite{S41}, now implies that 
the Fourier series of $f$ at $x$ converges to $f(x)$ 
for all $x\in D\setminus E_f$ with a null set $E_f\subset D$.
We put 
$D_0 := D \setminus E_0$, where
$
 E_0 := E \cup \bigcup_{f\in F_0} E_f
$
is a null set.
Then
for all $x\in D_0$ and $f\in F_0$, we have
\[
 K(x,x) \,< \infty
\quad \text{and} \quad
 f(x) \,=\, \sum_{k\in \N_0} \hat f(k)\, b_k (x).
\]

We now define the space $H$ as the set of all square-integrable functions $f\colon D_0 \to \C$
which are point-wise represented by their Fourier series
$\sum_k \hat f (k) \, b_k$
and which satisfy
\[
 \Vert f \Vert_H^2 \,:=\, \sum_{k\in\N_0} \frac{|\hat f(k)|^2}{\sigma_k^2} < \infty.
\]
Then $H$ is a separable reproducing kernel Hilbert space on $D_0$
since
\[
|f(x)|^2\leq K(x,x)\|f\|_H^2\quad\text{ for all }x\in D_0\text{ and }f\in H,
\]
and $(\sigma_k b_k)_{k\in\N_0}$ is an orthonormal basis of $H$.
The reproducing kernel is $K$, which
has finite trace from \eqref{eq:traceofH}.

We now show that $F_0$ (with functions restricted to $D_0$) is a subset of $H$.
Recall that any $f\in F_0$ is point-wise represented by its Fourier series. Moreover,
note that the Kolmogorov widths of $F_0$ and $F$ are the same.
We use 
\[
 d_{2m}(F) = \big(d_{2m}(F)^p\big)^{1/p} \le \bigg( \frac1m \sum_{k\ge m} d_k(F)^p \bigg)^{1/p}
\]
and obtain for any $m\in 8 \N$ and $f\in F_0$ that
\begin{equation*}
\begin{split}
 \Vert f - P_m f \Vert_H^2
 &= \sum_{k\ge m} k^{2\alpha} \, |\hat f(k)|^2 
	\le \sum_{\ell\in\N_0} \left(m2^{\ell+1}\right)^{2\alpha}  
	\sum_{k=m2^\ell}^{m2^{\ell+1}-1} |\hat f(k)|^2 \\
 &\le 4 \sum_{\ell\in\N_0} \left(m2^{\ell+1}\right)^{2\alpha}  d_{m 2^{\ell-2}}(F)^2\\
 &\le 4 \sum_{\ell\in\N_0} \left(m2^{\ell+1}\right)^{2\alpha} 
 \bigg( \frac{1}{m2^{\ell-3}} \sum_{k\ge m2^{\ell-3}} d_k(F)^p \bigg)^{2/p} \\
 &\leq\, 2^{2+2\alpha+6/p}m^{2\alpha-2/p}\sum_{\ell\in \N_0}2^{(2\alpha-2/p)\ell} 
	\bigg(\sum_{k\ge m/8} d_k(F)^p \bigg)^{2/p}.
\end{split}
\end{equation*}
The last expression is finite for $m\ge 8 k_0$,
since $2\alpha-2/p<0$.
This implies that $f\in H$ and
\begin{equation}
\label{eq1}
 \Vert f - P_m f \Vert_H
\,\leq\, C\, m^\alpha
	\bigg(\frac1m\sum_{k\ge m/8} d_k(F)^p \bigg)^{1/p} ,
\end{equation}
where $C>0$ only depends on $p\in(0,2)$ and $\alpha\in (\frac{1}{2},\frac{1}{p})$.

We now apply Theorem~\ref{main_theorem_local}
to the newly constructed Hilbert space $H$ to
find $n\le 43200\,m$ and a linear algorithm $A_n$ of the form
\[
 A_n(f) \,=\, \sum_{i=1}^n f(x_i) g_i, \quad x_i \in D_0, \ g_i \in L_2,
\]
such that
\begin{equation}\label{eq:xyz}
 \Vert f - A_n f \Vert_{L_2(D_0,\mu)}^2 
 \,\le\, 433 \,
\max\bigg\{\sigma_m^2, \, \frac{1}{m}\sum_{k\geq m}\sigma_k^2 \, \bigg\}
\, \Vert f - P_m f \Vert_{H}^2
\end{equation}
for all $f\in H$ and thus, for all $f\in F_0$.
Clearly, in the last inequality, $D_0$ can be replaced with $D$.
If we now insert the estimate~\eqref{eq1} and the estimate
\begin{equation}\label{eq2}
\max\bigg\{\sigma_m^2, \, \frac{1}{m}\sum_{k\geq m}\sigma_k^2 \, \bigg\}
 \,\lesssim\, m^{-2\alpha},
\end{equation}
into \eqref{eq:xyz}, we obtain that
\[
\Vert f-A_n f\Vert_{L_2}^2 
\,\le\, 
\bigg(\frac{\tilde c_p}{m} \sum_{k \ge m/8} d_k(F)^p\bigg)^{2/p}
\]
for all $f\in F_0$ and some $\tilde c_p>0$ that only depends on $p$.
Since $F_0$ is dense in $F$ and
both ${\rm id} \colon F \to L_2$ and $A_n \colon F \to L_2$
are continuous,
the last bound is true for all $f\in F$. 
This finishes the proof of Theorem~\ref{thm:non-Hilbert} 
with $c_p=43200 \, \max(\tilde c_p,8)$. \\ 
\end{proof}

\subsection{The boundary case}\label{sec:boundary}

We provide a variant of Theorem~\ref{thm:non-Hilbert} under a weaker 
condition than $(d_k(F))\in\ell_p$ for $p<2$.
In fact, we show that the condition
$\big((\log k)^s d_k(F)\big)\in\ell_2$ for some $s>1/2$
is enough for a comparison
of the sampling and the Kolmogorov widths,
while the same assumption for $s=1/2$ is not enough, see Example~\ref{ex2}.

\begin{theorem}\label{thm:non-Hilbert-boundary}
 Let $s>1/2$. There is a universal constant $c\in \N$
 and a constant $c_s>0$, depending only on $s$,
 such that for every $F$ and $\mu$ that satisfy \hyperlink{assum}{Assumption~A}
 and all $m\ge 2$,
 \[
 g_{cm}(F)^2 \,\le\, 
 c_s\, m^{-1} \log^{-2s+1} m \sum_{k \ge m} d_k(F)^2 \cdot \log^{2s} k.
 \]
\end{theorem}

\begin{proof}
The proof follows the same lines as the proof of Theorem~\ref{thm:non-Hilbert}.
The only difference is that we now choose $\sigma_k=k^{-1/2}\log^{-s}k$ for $k\ge 2$.
Then, inequality~\eqref{eq1} becomes
\[
\begin{split}
 \Vert f - P_m f \Vert_H^2
 &= \sum_{k\ge m} |\hat f(k)|^2 k\log^{2s}(k)
 \le \sum_{k\ge m} |\hat f(k)|^2 \sum_{m\le r\le 2k} \log^{2s}(r) \\
 &\le \sum_{r\ge m} \log^{2s}(r) \sum_{k\ge r/2}\,|\hat f(k)|^2
 \le 4 \sum_{r\ge m} \log^{2s}(r)\, d_{\lfloor r/8\rfloor}(F)^2 \\
 &\le 32 \sum_{k\ge \lfloor m/8 \rfloor} \log^{2s}(8k+7)\, d_k(F)^2.
\end{split}
\]
Likewise, inequality~\eqref{eq2} becomes
\[
\max\bigg\{\sigma_m^2, \, \frac{1}{m}\sum_{k\geq m}\sigma_k^2 \, \bigg\}
 \,\lesssim\, m^{-1} \log^{-2s+1} m
\]
and the stated inequality is obtained.
\end{proof}

\subsection{Proof of Corollary~\ref{cor:non-Hilbert}}

Using the same bound as in the proof of Corollary~\ref{cor:Hilbert}, 
the case $\alpha>1/2$ immediately follows from Theorem~\ref{thm:non-Hilbert}
if we choose $1/\alpha<p<2$, and the case $\alpha=1/2$, $\beta>1$ from Theorem~\ref{thm:non-Hilbert-boundary}
if we choose $1/2<s<\beta-1/2$.

All bounds are attained with the same classes $F$ as in Corollary~\ref{cor:Hilbert} for the first case,
and with the constructions from the next section for the two other cases. \\ \qed

\section{Examples} 
\label{sec:ex}

We first apply 
Theorem~\ref{main_theorem} 
to tensor product spaces.

\begin{example}\label{ex:tp}
Let $H$ be a RKHS on $D$ that is compactly embedded into $L_2$
and let $F$ be its unit ball.
We consider 
$L_2$-approximation on 
the unit ball $F_d$ of the $d$-fold tensor product $H_d$ of $H$,
which is a RKHS on the domain $D^d$.
We assume that $g_n(F) \lesssim n^{-\alpha}$ for some $\alpha>0$.
The famous Smolyak algorithm, see~\cite{Sm}, gives the estimate 
\begin{equation}\label{eq:smo}
 g_n(F_d) \,\lesssim\, n^{-\alpha} \log^{(\alpha+1)(d-1)} n.
\end{equation}
An example of such tensor product spaces are the spaces of dominating mixed smoothness $\alpha>1/2$, 
see~\cite{DTU}.
For these spaces, 
it is known that the error bound \eqref{eq:smo} for the Smolyak algorithm can be improved~\cite{SU};
the exponent of the logarithm can be reduced to $(\alpha+1/2)(d-1)$.
With Corollary~\ref{cor:Hilbert} and known results on the approximation numbers
of tensor product operators, see \cite{Bab,Mit}, we now obtain
\begin{equation}\label{eq:tp}
 g_n(F_d) \,\lesssim\, n^{-\alpha} \log^{\alpha(d-1)} n \quad \text{if } \ \alpha>1/2.
\end{equation}
This bound is asymptotically optimal for the spaces of mixed smoothness, 
see \cite[Theorem 1]{Tem92} or \cite[Theorem 6.4.3]{Tem18}. 
More generally,
it is known that $d_n(F) \asymp n^{-\alpha}$
implies $d_n(F_d) \asymp n^{-\alpha} \log^{\alpha(d-1)} n$ (see e.g.\ \cite{K18})
and therefore the asymptotic bound~\eqref{eq:tp} 
is optimal whenever 
the approximation numbers in the univariate case are of order $n^{-\alpha}$.
Let us note, however, that
also preasymptotic estimates on the sampling numbers (say, for $n<d^d$)
are of interest, especially if the dimension $d$ is high,
see \cite{K18,KSU,WW95}. \\
\qed
\end{example}

We now present two examples that show that 
our upper bounds 
cannot be improved without further assumptions 
on the class $F$.

First, we show that the worst possible behavior of the sampling numbers
in the case $d_n(F) \lesssim n^{-1/2} \log^{-\beta} n$ with $\beta>1$
is indeed $n^{-1/2} \log^{-\beta+1} n$.

\begin{example}\label{ex1}
For $\ell\in\N_0$ and $k\in\{1,\hdots,2^\ell\}$ define the interval $I_{\ell,k} = [(k-1)2^{-\ell},k2^{-\ell})$
and denote with $\chi_{\ell,k}$ the indicator function of $I_{\ell,k}$.
Let $\beta>1$.
We set
\[
 \mathcal C_\beta \,:=\, \bigg\{ \mathbf c := (c_{\ell,k})_{\ell\in\N_0,\,1\le k\le 2^\ell} 
 \,\Big\vert\, \sum_{k=1}^{2^\ell} |c_{\ell,k}|^2 \le (\ell+1)^{-2\beta} \text{ for all } \ell\in\N_0 \bigg\} 
\]
and consider the class
\[
 F_\beta \,:=\, \bigg\{ f_{\mathbf c} := \sum_{\ell\in\N_0} \sum_{k=1}^{2^\ell} c_{\ell,k} \chi_{\ell,k}  
 \,\Big\vert\, \mathbf c \in \mathcal C_\beta  \bigg\}.
\]
Note that the series $f_{\mathbf c}$ converge uniformly, since the inner sum is bounded by $(\ell+1)^{-\beta}$.
If $F_\beta$ is equipped with the maximum distance on $[0,1)$,
it is a separable 
metric space,
function evaluation is continuous,
and the embedding in $L_2([0,1))$ is continuous.

For every $L\in\N_0$, the span $V_L$ of the functions  $\chi_{\ell,k}$ with $\ell \le L$
has dimension $2^{L}$.
If $P_L$ is the $L_2$-orthogonal projection onto $V_L$, we have for all $\mathbf c \in \mathcal C_\beta$ that
\begin{multline*}
 \big\Vert f_{\mathbf c} - P_L f_{\mathbf c}\big\Vert_2
 \le \Big\Vert \sum_{(\ell,k) \colon \ell > L} c_{\ell,k} \chi_{\ell,k} \Big\Vert_2
 \le \sum_{\ell > L} \Big\Vert \sum_{k=1}^{2^\ell} c_{\ell,k} \chi_{\ell,k} \Big\Vert_2 \\
 = \sum_{\ell > L} \bigg(\sum_{k=1}^{2^\ell} c_{\ell,k}^2 \Vert\chi_{\ell,k}\Vert_2^2 \bigg)^{1/2}
 \le \sum_{\ell > L} 2^{-\ell/2} (\ell+1)^{-\beta}
 \lesssim 2^{-L/2} L^{-\beta},
\end{multline*}
and thus
\[
  d_{2^L}(F_\beta) \,\lesssim\, 2^{-L/2} L^{-\beta},
\]
or equivalently
\[
  d_n(F_\beta) \,\lesssim\, n^{-1/2} \log^{-\beta} n.
\]
We now show a lower bound for the sampling numbers.
Let $x_1,\hdots,x_n\in [0,1)$.
For all $\ell\in\N_0$, we let $J_\ell$ be the set of indices $1\le k\le 2^{\ell}$ such that $I_{\ell,k}$ contains at least one of these points.
Clearly, the cardinality of $J_\ell$ is at most $n$.
We choose $L\in \N_0$ of order $\log n$ and define
\[
 f_L \,:=\, \sum_{\ell>L}  |J_\ell|^{-1/2} (\ell+1)^{-\beta}\sum_{k\in J_\ell} \chi_{\ell,k}. 
\]
This function is contained in $F_\beta$
and for all $i\le n$, we have
\[
 h \,:=\, f_L(x_i) 
 \,=\, \sum_{\ell>L} |J_\ell|^{-1/2} (\ell+1)^{-\beta}
 \,\gtrsim\, n^{-1/2} \log^{-\beta+1} n,
\]
where $h$ is independent of $i$.
On the other hand, as shown by our previous calculation,
\[
 \Big|\int_0^1 f_L(x)\,dx\Big| \,\le\, \big\Vert f_L \big\Vert_2
 \,\lesssim\, 2^{-L/2} L^{-\beta}
 \,\lesssim\, n^{-1/2} \log^{-\beta} n.
\]
Thus, if we set $f = h - f_L$, the function is contained in $F_\beta$,
vanishes at all points $x_1,\hdots,x_n$,
and satisfies
\[
 \Vert f \Vert_2
 \,\ge\,
 \int_0^1 f(x)\,dx \,\ge\, h - \Big|\int_0^1 f_L(x)\,dx\Big|
 \,\gtrsim\, n^{-1/2} \log^{-\beta+1} n.
\]
This shows $g_n(F_\beta)\gtrsim n^{-1/2} \log^{-\beta+1} n$. \\ \qed
\end{example}

The next example shows that, in the case $d_n(F) \lesssim n^{-1/2} \log^{-\beta} n$
with $\beta\le 1$, no general statement on the sampling numbers is possible.

\begin{example}\label{ex2}
Similar to Example~\ref{ex1}, we define
\[
 \mathcal C \,:=\, \bigg\{ \mathbf c 
 \,\,\Big\vert\, \sum_{k=1}^{2^\ell} |c_{\ell,k}|^2 \le (\ell+1)^{-2} \log(\ell+e)^{-2} \text{ for all } \ell\in\N_0 \bigg\} 
\]
and consider the class
\[
 F \,:=\, \bigg\{ f_{\mathbf c} := \sum_{\ell\in\N_0} \sum_{k=1}^{2^\ell} c_{\ell,k} \chi_{\ell,k}  
 \,\Big\vert\, \mathbf c \in \mathcal C, \, \mathbf c \text{ finite}  \bigg\}.
\]
The finiteness of the sequences ensures that
$F$, equipped with the maximum distance, is still a separable metric space, 
where function evaluation is continuous,
and the embedding in $L_2([0,1))$ is continuous.
As above, we obtain
\[
  d_n(F) \,\lesssim\, n^{-1/2} (\log n)^{-1} (\log\log n)^{-1}.
\]
In particular, we have $(d_n(F) \log^{1/2} n) \in \ell_2$.
On the other hand, given $x_1,\hdots,x_n$ and $\varepsilon>0$, 
we choose $L\in\N_0$ with
\[
\sum_{\ell > L} 2^{-\ell/2} (\ell+1)^{-1} (\log(\ell+e))^{-1} \,\le\, \varepsilon,
\]
define the sets $J_\ell$ as above,
and choose $N\in\N_0$ such that
\[
 h \,:=\, 
 \sum_{\ell=L+1}^N |J_\ell|^{-1/2}(\ell+1)^{-1} (\log(\ell+e))^{-1} \,\ge\, 1.
\] 
The function 
\[
 f_L \,:=\, \frac{1}{h} \sum_{\ell=L+1}^{N} |J_\ell|^{-1/2} (\ell+1)^{-1} (\log(\ell+e))^{-1} \sum_{k\in J_\ell} \chi_{\ell,k},
\]
is contained in $F$, its integral 
is at most $\varepsilon$, 
and it satisfies $f_L(x_i)=1$ for all $i\le n$.
Then $f = 1 - f_L$ is contained in $F$,
vanishes at all points $x_1,\hdots,x_n$,
and satisfies
\[
 \Vert f \Vert_2
 \,\ge\,
 \int_0^1 f(x)\,dx \,\ge\, 1 - \Big|\int_0^1 f_L(x)\,dx\Big|
 \,\ge\, 1-\varepsilon.
\]
This shows $g_n(F)\ge 1$ for all $n\in\N_0$.\\ \qed
\end{example}

We note that the lower bounds in Example~\ref{ex1} and \ref{ex2}
already hold for the easier problem of numerical integration on $F_\beta$.
Thus, the upper bounds from Corollary~\ref{cor:non-Hilbert}
are also sharp for the minimal error of quadrature rules on probability spaces.


\end{document}